\documentclass[11pt, a4paper]{article}
\usepackage{amsmath, amsthm, amssymb, url}
\usepackage[margin=2cm]{geometry}
\usepackage{multirow}

\newcommand{\Tran}{\mathrm{Tran}}

\newcommand{\Sym}{\mathrm{Sym}}
\newcommand{\Alt}{\mathrm{Alt}}

\newcommand{\id}{\mathrm{id}}

\newcommand{\CA}{\mathrm{CA}}
\newcommand{\ICA}{\mathrm{ICA}}
\newcommand{\Rank}{\mathrm{Rank}}
\newcommand{\di}{\mathrm{d}}

\theoremstyle{plain}

\newtheorem{corollary}{Corollary}
\newtheorem{lemma}{Lemma}

\newtheorem{theorem}{Theorem}
\newtheorem*{claim*}{Claim}

\theoremstyle{definition}
\newtheorem{definition}{Definition}
\newtheorem{problem}{Problem}
\newtheorem{example}{Example}
\newtheorem{remark}{Remark}

\begin{document}

\title{Ranks of finite semigroups of one-dimensional cellular automata}
\author{Alonso Castillo-Ramirez\footnote{Corresponding author. Email: \texttt{alonso.castillo-ramirez@durham.ac.uk}} \ and Maximilien Gadouleau\footnote{Email: \texttt{m.r.gadouleau@durham.ac.uk}}  \\ \\ 
\small School of Engineering and Computing Sciences, \\ 
\small Durham University, South Road, \\ 
\small Durham, DH1 3LE \\
\small Telephone: +44 (0) 191 33 41729}
\maketitle

\begin{abstract}
Since first introduced by John von Neumann, the notion of cellular automaton has grown into a key concept in computer science, physics and theoretical biology. In its classical setting, a cellular automaton is a transformation of the set of all configurations of a regular grid such that the image of any particular cell of the grid is determined by a fixed local function that only depends on a fixed finite neighbourhood. In recent years, with the introduction of a generalised definition in terms of transformations of the form $\tau : A^G \to A^G$ (where $G$ is any group and $A$ is any set), the theory of cellular automata has been greatly enriched by its connections with group theory and topology. In this paper, we begin the finite semigroup theoretic study of cellular automata by investigating the rank (i.e. the cardinality of a smallest generating set) of the semigroup $\CA(\mathbb{Z}_n; A)$ consisting of all cellular automata over the cyclic group $\mathbb{Z}_n$ and a finite set $A$. In particular, we determine this rank when $n$ is equal to $p$, $2^k$ or $2^kp$, for any odd prime $p$ and $k \geq 1$, and we give upper and lower bounds for the general case. 
\end{abstract}


\section{Introduction}

Cellular automata (CA) were introduced by John von Neumann as an attempt to design self-reproducing systems that were computationally universal (see \cite{N66}). Since then, the theory of CA has grown into an important area of computer science, physics, and theoretical biology (e.g. \cite{CSC10,Ka05,Wo84}). Among the most famous CA are Rule 110 and John Conway's Game of Life, both of which have been widely studied as discrete dynamical systems and are known to be capable of universal computation.   

In recent years, many interesting results linking CA and group theory have appeared in the literature (e.g. see \cite{B10,CSC10,CSMS99}). One of the goals of this paper is to embark in the new task of exploring CA from the point of view of finite semigroup theory.

We shall review the broad definition of CA that appears in \cite[Sec.~1.4]{CSC10}. Let $G$ be a group and $A$ a set. Denote by $A^G$ the set of functions of the form $x:G \to A$. For each $g \in G$, denote by $R_g : G \to G$ the right multiplication map, i.e. $(h)R_g := hg$ for any $h \in G$. We shall emphasise that in this paper we apply maps on the right, while in \cite{CSC10} maps are applied on the left.   

\begin{definition} \label{def:ca}
Let $G$ be a group and $A$ a set. A \emph{cellular automaton} over $G$ and $A$ is a map $\tau : A^G \to A^G$ satisfying the following property: there exists a finite subset $S \subseteq G$ and a \emph{local map} $\mu : A^S \to A$ such that 
\[ (g)(x)\tau = (( R_g \circ x  )\vert_{S}) \mu, \]
for all $x \in A^G$, $g \in G$, where $(R_g \circ x )\vert_{S}$ is the restriction to $S$ of $(R_g \circ x) : G \to A$ .
\end{definition}

Let $\CA(G;A)$ be the set of all cellular automata over $G$ and $A$; it is straightforward to show that, under composition of maps, $\CA(G;A)$ is a semigroup. Most of the literature on CA focus on the case when $G=\mathbb{Z}^d$, $d\geq1$, and $A$ is a finite set (see \cite{Ka05}). In this situation, an element $\tau \in \CA(\mathbb{Z}^d;A)$ is referred as a \emph{$d$-dimensional} cellular automaton. 

Although results on semigroups of CA have appeared in the literature before (see \cite{H12,S15}), the semigroup structure of $\CA(G;A)$ remains basically unknown. In particular, the study of the finite semigroups $\CA(G;A)$, when $G$ and $A$ are finite, has been generally disregarded, perhaps because some of the classical questions are trivially answered (e.g. the Garden of Eden theorem becomes trivial). However, many new questions, typical of finite semigroup theory, arise in this setting. 

One of the fundamental problems in the study of a finite semigroup $M$ is the determination of the cardinality of a smallest generating subset of $M$; this is called the \emph{rank} of $M$ and denoted by $\Rank(M)$:
\[ \Rank(M) := \min \{ \vert H \vert :  H \subseteq M \text{ and } \langle H \rangle = M \}. \]
It is well-known that, if $X$ is any finite set, the rank of the full transformation semigroup $\Tran(X)$ (consisting of all functions $f : X \to X$) is $3$, while the rank of the symmetric group $\Sym(X)$ is $2$ (see \cite[Ch.~3]{GM09}). Ranks of various finite semigroups have been determined in the literature before (e.g. see \cite{ABJS14,AS09,GH87,G14,HM90}). 

In order to hopefully bring more attention to the study of finite semigroups of CA, we shall propose the following problem.
\begin{problem}\label{problem}
For any finite group $G$ and any finite set $A$, determine $\Rank(\CA(G;A))$.
\end{problem}                      
A natural restriction of this problem, and perhaps a more feasible goal, is to determine the ranks of semigroups of CA over finite abelian groups.

In this paper we study the finite semigroups $\CA(\mathbb{Z}_n;A)$, where $\mathbb{Z}_n$ is the cyclic group of order $n \geq 2$ and $A$ is a finite set with at least two elements. By analogy with the classical setting, we may say that the elements of $\CA(\mathbb{Z}_n;A)$ are one-dimensional cellular automata over $\mathbb{Z}_n$ and $A$.

In this paper we shall establish the following theorems.

\begin{theorem}\label{main1}
Let $k \geq 1$ be an integer, $p$ an odd prime, and $A$ a finite set of size $q \geq 2$. Then:
\begin{align*}
\Rank(\CA(\mathbb{Z}_p; A)) &= 5;  \\ 
\Rank(\CA(\mathbb{Z}_{2^k}; A)) &=  \begin{cases}
\frac{1}{2} k (k+7), & \text{ if } q = 2; \\  
\frac{1}{2} k (k+7) + 2, & \text{ if } q \geq 3;
\end{cases} \\ 
\Rank(\CA(\mathbb{Z}_{2^kp}; A)) & = \begin{cases}
\frac{1}{2} k (3k+17) + 3 , & \text{ if } q=2; \\
\frac{1}{2} k (3k+ 17) + 5, & \text{ if } q \geq 3. 
\end{cases}
\end{align*}
\end{theorem}

Let $2 \mathbb{Z}$ be the set of even integers. For any integer $n \geq 2$, let $[n]:=\{1,2, \dots, n\}$. Denote by $\di(n)$ the number of divisors of $n$ (including $1$ and $n$ itself) and by $\di_+(n)$ the number of even divisors of $n$. Let
\[ E(n) := \left\vert \left\{ (s,t) \in [n]^2: t\mid n, \ s \mid n, \text{ and } t \mid s \right\} \right\vert \]
be the number of edges in the \emph{divisibility digraph} of $n$ (see Section \ref{relative rank}). 

\begin{theorem}\label{main2}
Let $n \geq 2$ be an integer and $A$ a finite set of size $q \geq 2$. Then:
\[  \Rank( \CA(\mathbb{Z}_n; A)  ) =  \begin{cases}
\di(n) + \di_+(n) + E(n) - 2 + \epsilon(n,2), & \text{ if } q=2 \text{ and } n \in 2 \mathbb{Z}; \\ 
\di(n) + \di_+(n) + E(n) + \epsilon(n,q), & \text{otherwise;} 
\end{cases} \] 
where $0 \leq \epsilon(n,q) \leq \max\{ 0 , \di(n) - \di_+(n) - 2 \}$.
\end{theorem}


\section{Preliminary results}

For the rest of the paper, let $n \geq 2$ an integer and $A$ a finite set of size $q \geq 2$. We may assume that $A = \{0,1,\dots,q-1\}$. When $G$ is a finite group, we may always assume that the finite subset $S \subseteq G$ of Definition \ref{def:ca} is equal to $G$, so any cellular automaton over $G$ and $A$ is completely determined by the local map $\mu: A^G \to A$. Therefore, if $\vert G \vert = n$, we have $\left\vert  \CA(G ; A)\right\vert = q^{q^n}$.

 It is clear that $\CA(\mathbb{Z}_n;A)$ is contained in the semigroup of transformations $\Tran(A^n)$, where $A^n$ is the $n$-th Cartesian power of $A$. For any $f \in \Tran(A^n)$ write $f=(f_1, \dots, f_n)$, where $f_i : A^n \to A$ is the \emph{$i$-th coordinate function of $f$}. For any semigroup $M$ and $\sigma \in M$, define the \emph{centraliser} of $\sigma$ in $M$ by
\[ C_{M} (\sigma) := \{ \tau \in M : \tau \sigma = \sigma \tau\}. \]
It turns out that $\CA(\mathbb{Z}_n ; A)$ is equal to the centraliser of a certain transformation in $\Tran(A^n)$.

For any $f \in \Tran(A^n)$, define an equivalence relation $\sim$ on $A^n$ as follows: for any $x,y \in A^n$, say that $x \sim y$ if and only if there exist $r,s \geq 1$ such that $(x)f^r = (y)f^s$. The equivalence classes induced by this relation are called the \emph{orbits} of $f$. 

\begin{lemma} \label{le:basic}
Let $n \geq 2$ be an integer and $A$ a finite set. Consider the map $\sigma : A^n \to A^n$ given by
\[ (x_1, \dots, x_n ) \sigma = (x_n, x_1, \dots, x_{n-1}). \]
Then:
\begin{description}
\item[(i)] $\CA(\mathbb{Z}_n ; A) = C_{\Tran(A^n)}(\sigma) := \{ \tau \in \Tran(A^n) : \tau \sigma = \sigma \tau \}$.
\item[(ii)] Let $\mathcal{O}$ be the set of orbits of $\sigma : A^n \to A^n$. For every $P \in \mathcal{O}$, $\vert P \vert$ divides $n$.
\item[(iii)] Every $\tau \in \CA(\mathbb{Z}_n; A)$ satisfies the following property: for every $P \in \mathcal{O}$ there exists $Q \in \mathcal{O}$, with $\vert Q \vert$ dividing $\vert P \vert$, such that $(P)\tau = Q$.
\end{description}
\end{lemma}
\begin{proof}
We shall prove each part.
\begin{description}
\item[(i)] By Definition \ref{def:ca}, a map $\tau:A^n \to A^n$ is a cellular automaton over $G=\mathbb{Z}_n$ and $A$ if and only if there exists a map $\mu :A^n \to A$ such that
\[ (x_1, x_2, \dots, x_n )\tau_i = (x_{1+i},x_{2+i}, \dots, x_{n+i})\mu \]
for any $1 \leq i \leq n$, where the sum in the subindex of $x_{j+i}$ is done modulo $n$. Hence,
\begin{align*}
(x_1, x_2, \dots, x_n) \sigma \tau &= (x_n, x_1, \dots, x_{n-1})\tau \\
& = ((x_1, x_2, \dots, x_n )\mu, (x_2, x_3, \dots, x_1)\mu, \dots, (x_n, x_1, \dots, x_{n-1})\mu) \\
& = ((x_2, x_3, \dots, x_1)\mu, (x_3, x_4, \dots, x_2) \mu, \dots, (x_1, x_2, \dots, x_n )\mu ) \sigma \\
& = (x_1, x_2, \dots, x_n) \tau \sigma.
\end{align*}
This shows that $\CA(\mathbb{Z}_n ; A) \leq \{ \tau \in \Tran(A^n) : \tau \sigma = \sigma \tau \}$. Let $f \in \Tran(A^n)$ be such that $f\sigma = \sigma f$. This implies that $f \sigma^k = \sigma^k f$ for any $k \in \mathbb{Z}$, so
\[ (x_1, x_2, \dots, x_n ) f_{n-k} = (x_{1-k}, x_{2-k}, \dots, x_{n-k}) f_{n}. \]
Therefore, $f$ is a cellular automaton over $\mathbb{Z}_n$ and $A$ with $\mu = f_n$.

\item[(ii)] This follows directly by the Orbit-Stabiliser Theorem (\cite[Theorem 1.4A]{Dixon}).

\item[(iii)] Fix $\tau \in \CA(\mathbb{Z}_n ; A)$, $P \in \mathcal{O}$ and $x \in \mathcal{O}$. By definition of orbit, and since $\sigma$ is a permutation, for every $y \in P$ there is $i \in \mathbb{Z}$ such that $(x)\sigma^i = y$. By part (i), $(x)\tau \sigma^i = (x)\sigma^i \tau = (y)\tau$, so $(P)\tau \subseteq Q$ for some $Q \in \mathcal{O}$. Furthermore, for every $z \in Q$ there is $j \in\mathbb{Z}$ such that $(z)\sigma^j = (x)\tau$, so $z = (x)\sigma^{-j}\tau \in (P)\tau$. This shows that $(P)\tau = Q$. Finally, we show that $\vert Q \vert$ divides $\vert P \vert$. Fix $z \in Q$. For any $w \in Q$ there is $k \in \mathbb{Z}$ such that $z = (w) \sigma^k$. Then $\sigma^k$ is a bijection between the preimage sets $(z) \tau^{-1}$ and $(w) \tau^{-1}$. This means that $\left\vert (z) \tau^{-1} \right\vert = \left\vert (w) \tau^{-1} \right\vert$ for every $w \in Q=(P)\tau$. Therefore, 
\[ \vert P \vert = \sum_{w \in Q} \left\vert (w)\tau^{-1} \right\vert =  \left\vert (z)\tau^{-1} \right\vert \cdot \left\vert Q \right\vert.\] 
\end{description}
\end{proof}

Lemma \ref{le:basic} \textbf{(i)} is in fact a particular case of a more general result.

\begin{lemma}
Let $G$ be a finite group and $A$ a finite set. For each $g \in G$, let $\sigma_g \in \CA(G;A)$ be the cellular automaton with local map $\mu_g : A^G \to A$ defined by $(x)\mu_g = (g^{-1})x$ for all $x \in A^G$. Then,
\[ \CA(G;A) = \{ \tau : A^G \to A^G : \tau \sigma_g = \sigma_g \tau, \ \forall g \in G   \}. \]
\end{lemma}
\begin{proof}
The result follows by Curtis-Hedlund Theorem (see \cite[Theorem 1.8.1]{CSC10}).
\end{proof}

Let $\ICA(G;A)$ be the set of invertible cellular automata:
\[ \ICA(G;A) := \{ \tau \in \CA(G;A) : \exists \phi \in \CA(G;A) \text{ such that } \tau \phi = \phi \tau = \id \}. \]
It may be shown that $\ICA(G;A) = \CA(G;A) \cap \Sym(A^G)$ whenever $A$ is finite (see \cite[Theorem 1.10.2]{CSC10}).

We shall use the cyclic notation to denote the permutations in $\Tran(A^n)$. If $D \subseteq A^n$ and $a \in A^n$, we define the transformation $(D \to a) \in \Tran(A^n)$ by
\[ (x)(D \to a) := \begin{cases}
a & \text{ if } x \in D \\
x & \text{ otherwise }. 
 \end{cases} \]
 When $D=\{ b\}$ is a singleton, we write $(b \to a)$ instead of $(\{ b\} \to a)$.
 
 In the following examples, we identify the elements of $A^n$ with their lexicographical order: $(a_1, a_2, \dots, a_n) \sim \sum_{i=1}^n a_i q^{i-1}$.

\begin{example}\label{ex:p=q=2}
 A generating set for $\CA \left( \mathbb{Z}_{2} ; \{ 0,1 \} \right)$ is 
\[ \{ (1,2), \ (\{1,2\} \to 0), \ (0,3), \ (3 \to 0) \}, \]
where $0:=(0,0)$, $1:=(1,0)$, $2:=(0,1)$ and $3:=(1,1)$. Direct calculations in GAP show that indeed $\Rank(\CA \left( \mathbb{Z}_{2} ; \{ 0,1 \} \right)) = 4$.
\end{example}

\begin{example}
A generating set for $\CA \left( \mathbb{Z}_{3} ; \{ 0,1 \} \right)$ is 
\[ \{\left( 1,2,4\right)\left( 0,7\right), \ \left( 1,6\right) \left( 2,5\right) \left(3,4\right), \ (1 \to 6) (2 \to 5)( 4 \to 3), \  (\{1,2,4\} \to 0), \  \left( 7\rightarrow 0\right)   \}. \]
Direct calculations in GAP show that indeed $\Rank(\CA \left( \mathbb{Z}_{3} ; \{ 0,1 \} \right)) = 5$.
\end{example}

If $U$ is a subset of a finite semigroup $M$, the \emph{relative rank} of $U$ in $M$, denoted by $\Rank(M:U)$, is the minimum cardinality of a subset $V \subseteq M$ such that $\langle U,V \rangle = M$. The proof of the main results of this paper are based in the following observation.

\begin{lemma} \label{le:preliminar}
Let $G$ be a finite group and $A$ a finite set. Then,
\[ \Rank(\CA(G ; A)) = \Rank(\CA(G;A):\ICA(G;A)) + \Rank(\ICA(G;A)). \]
\end{lemma}
\begin{proof}
As $\ICA(G;A)$ is the group of units of $\CA(G ; A)$, this follows by \cite[Lemma 3.1]{AS09}.
\end{proof}

In Section \ref{rank ica} we study the rank of $\ICA(\mathbb{Z}_n;A)$, while in Section \ref{relative rank} we study the relative rank of $\ICA(\mathbb{Z}_n;A)$ in $\CA(\mathbb{Z}_n;A)$.


\section{The rank of $\ICA(\mathbb{Z}_n ; A)$} \label{rank ica}

Let $\sigma : A^n \to A^n$ be as defined in Lemma \ref{le:basic}. For any $d \geq 1$ dividing $n$, the number of orbits of $\sigma$ of size $d$ is given by the Moreau's necklace-counting function
\[ \alpha(d, q) = \frac{1}{d} \sum_{b \mid d} \mu \left( \frac{d}{b} \right) q^{b}, \]
where $\mu$ is the classic M\"obius function (see \cite{MR83}). In particular, if $d=p^k$, where $p$ is a prime number and $k \geq 1$, then
\begin{equation}\label{eq:Moreau}
\alpha(p^k, q) = \frac{q^{p^k} - q^{p^{k-1}}}{p^k}.
\end{equation}

\begin{remark}
Observe that $\alpha(d,q) = 1$ if and only if $(d,q)=(2,2)$. Hence, the case when $n$ is even and $q=2$ is degenerate and shall be analysed separately in the rest of the paper. 
\end{remark}

We say that $d$ is a non-trivial divisor of $n$ if $d \mid n$ and $d \neq 1$ (note that, in our definition, $d = n$ is a non-trivial divisor of $n$). For any integer $\alpha \geq 1$, let $\Sym_\alpha$ and $\Alt_\alpha$ be the symmetric and alternating groups on $[\alpha]=\{1, \dots, \alpha \}$, respectively. 

A wreath product of the form $\mathbb{Z}_d \wr \Sym_{\alpha} := \{ (v; \phi) : v \in (\mathbb{Z}_d)^\alpha, \phi \in \Sym_\alpha \}$ is called a \emph{generalized symmetric group} (see \cite{O54}). We shall use the additive notation for the elements of $(\mathbb{Z}_d)^\alpha$, so the product in $\mathbb{Z}_d \wr \Sym_{\alpha}$ is 
\[ (v;\phi) \cdot (w; \psi) = (v + w^{\phi}; \phi \psi), \]
where $v,w \in (\mathbb{Z}_d)^\alpha$, $\phi, \psi \in \Sym_\alpha$, and $\phi$ acts on $w$ by permuting the coordinates. We shall identify the elements $(v; \id) \in \mathbb{Z}_d \wr \Sym_{\alpha}$ with $v \in (\mathbb{Z}_d)^\alpha$.

The following result is a refinement of \cite[Theorem 9]{S15}.

\begin{lemma} \label{le:ICA iso}
Let $n \geq 2$ be an integer and $A$ a finite set of size $q \geq 2$. Let $d_1, d_2, \dots, d_\ell$ be the non-trivial divisors of $n$. Then
\[ \ICA(\mathbb{Z}_n; A) \cong (\mathbb{Z}_{d_1} \wr \Sym_{\alpha(d_1,q)}) \times \dots \times (\mathbb{Z}_{d_\ell} \wr \Sym_{\alpha(d_\ell,q)}) \times \Sym_q .  \]
\end{lemma}
\begin{proof}
Let $\mathcal{O}$ the set of orbits of $\sigma : A^n \to A^n$ as defined in Lemma \ref{le:basic}. Part \textbf{(ii)} of that lemma shows that $\CA(\mathbb{Z}_n ; A)$ is contained in the semigroup
\[ \Tran(A^n, \mathcal{O}) := \{ f \in \Tran(A^n) : \forall P \in \mathcal{O}, \ \exists Q \in \mathcal{O} \text{ such that } (P)f \subseteq Q \}. \]
As $\mathcal{O}$ contains $q$ singletons and $\alpha(d_i, q)$ orbits of size $d_i \geq 2$, we know by \cite[Lemma 2.1]{ABJS14} that the group of units of $\Tran(A^n, \mathcal{O})$ is
\[ S(A^n, \mathcal{O}) \cong (\Sym_{d_1} \wr \Sym_{\alpha(d_1,q)}) \times \dots \times (\Sym_{d_\ell} \wr \Sym_{\alpha(d_\ell,q)}) \times \Sym_q.   \]
Clearly, $\ICA(\mathbb{Z}_n ; A) \leq  S(A^n, \mathcal{O})$. Let $P$ be an orbit of size $d_i$. Since the restriction of $\sigma$ to $P$, denoted by $\sigma \vert_{P}$, is a cycle of length $d_i$, and the centraliser of $\sigma \vert_{P}$ in $\Sym_{d_i}$ is $\langle \sigma \vert_{P} \rangle \cong \mathbb{Z}_{d_i}$, it follows that
\[ \ICA(\mathbb{Z}_n ; A) \leq (\mathbb{Z}_{d_1} \wr \Sym_{\alpha(d_1,q)}) \times \dots \times (\mathbb{Z}_{d_\ell} \wr \Sym_{\alpha(d_\ell,q)}) \times \Sym_q. \]
Equality follows as any permutation stabilising the sets of orbits of size $d_i$ commutes with $\sigma$. 
\end{proof}

For $1 \leq i \leq \alpha$, denote by $e^i$ the element of $(\mathbb{Z}_d)^\alpha$ with $1$ at the $i$-th coordinate, and $0$ elsewhere. Denote by $e^0$ the element of $(\mathbb{Z}_d)^\alpha$ with $0$'s everywhere. For any $\alpha \geq 2$, define permutations $z_\alpha \in \Sym_\alpha$ by 
\begin{equation} \label{permz}
z_\alpha := \begin{cases}
(1,2,3,\dots,\alpha), & \text{ if } \alpha \text{ is odd,} \\
(2,3, \dots, \alpha) & \text{ if } \alpha \text{ is even.}
\end{cases}
\end{equation}
Note that the order of $z_\alpha$, denoted by $o(z_\alpha)$, is always odd.

In the following Lemma we determine the rank of the generalized symmetric group.

\begin{lemma} \label{le:rk2}
Let $d, \alpha \geq 2$. Then, $\Rank \left( \mathbb{Z}_d \wr \Sym_{\alpha} \right) = 2$.
\end{lemma}
\begin{proof}
It is clear that $\mathbb{Z}_d \wr \Sym_{\alpha}$ is not a cyclic group, so $2 \leq \Rank \left( \mathbb{Z}_d \wr \Sym_{\alpha} \right)$. 

Define $z_\alpha$ as in (\ref{permz}). We will show that $\mathbb{Z}_d \wr \Sym_{\alpha}$ is generated by 
\[ x:=  (e^1; z_\alpha) \ \text{ and } \ y:=  (e^1 ;(1,2)). \]

Let $M := \langle x,y \rangle \leq \mathbb{Z}_d \wr \Sym_{\alpha}$. Let $\rho : \mathbb{Z}_d \wr \Sym_{\alpha} \to \Sym_\alpha$ be the natural projection, and note that this is a group homomorphism. Clearly, $(M)\rho = \Sym_\alpha$ and $\ker(\rho) = (\mathbb{Z}_d)^\alpha$, so, in order to prove that $M=\mathbb{Z}_d \wr \Sym_{\alpha}$, it suffices to show that $(\mathbb{Z}_d)^\alpha \leq M$.

Since $(M)\rho = \Sym_\alpha$, the intersection $(\mathbb{Z}_d)^\alpha \cap M$ is a $\Sym_\alpha$-invariant submodule of $(\mathbb{Z}_d)^\alpha$. Observe that
\[ y^2 = e^1 + e^2 =  (1, 1, 0 \dots, 0) \in (\mathbb{Z}_d)^\alpha \cap M. \]
Now, by $\Sym_\alpha$-invariance 
\begin{align*}
 & y^2 + \sum_{i=1}^{d-1} (y^2)^{(1,2,3)} + (y^2)^{(1,3,2)}  \\
 =&  (1, 1, 0,  \dots, 0) + (0, d-1, d-1, 0, \dots, 0) + (1,0,1,0, \dots, 0)  \\
  = &   (2,0,\dots,0) =: 2e^1 \in (\mathbb{Z}_d)^\alpha \cap M 
\end{align*}
If $d$ is odd, then $2e^1$ generates $(\mathbb{Z}_d)^\alpha$ as $\Sym_\alpha$-module, so $(\mathbb{Z}_d)^\alpha \cap M = (\mathbb{Z}_d)^\alpha$. 

Suppose that $d$ is even and $\alpha$ is odd. Then,
\[ x^{\alpha} = (1,1, \dots, 1) \in (\mathbb{Z}_d)^\alpha \cap M.\] 
Since $\Sym_\alpha$ is $2$-transitive on the basis of $(\mathbb{Z}_d)^\alpha$ and $y^2 = (1, 1, 0 \dots, 0) \in (\mathbb{Z}_d)^\alpha \cap M$, we obtain that $(1,\dots,1,0) \in (\mathbb{Z}_d)^\alpha \cap M$. Therefore, 
\[ (1,1, \dots, 1) - (1,\dots,1,0) = (0,\dots,0,1) \in (\mathbb{Z}_d)^\alpha \cap M,\]
and $(\mathbb{Z}_d)^\alpha \cap M = (\mathbb{Z}_d)^\alpha$.

Finally, suppose that $d$ and $\alpha$ are both even. Then,
\[ x^{\alpha-1} = (\alpha - 1, 0, \dots, 0) \in (\mathbb{Z}_d)^\alpha \cap M.  \]
Write $\alpha-1 = 2k + 1$, for some $k \in \mathbb{N}$. Then
\[  x^{\alpha-1} - \sum_{i=1}^k 2e^1 =  (1,0, \dots, 0) \in (\mathbb{Z}_d)^\alpha \cap M.  \] 
Therefore, $(\mathbb{Z}_d)^\alpha \cap M = (\mathbb{Z}_d)^\alpha$.
\end{proof}

We need the following results in order to establish $\Rank(\ICA(\mathbb{Z}_p, A))$ when $p$ is a prime number. 

\begin{lemma}[Lemma 5.3.4 in \cite{KL90}] \label{le:sub}
Let $\alpha \geq 2$. The permutation module for $\Sym_\alpha$ over a field $\mathbb{F}$ of characteristic $p$ has precisely two proper nonzero submodules:
\[ U_1 := \left\{ (a,a, \dots, a) : a \in \mathbb{F} \right\} \ \ \text{and} \ \  U_2 := \left\{ (a_1, a_2, \dots, a_\alpha) \in \mathbb{F}^{\alpha} : \sum_{i=1}^\alpha a_i = 0 \right\}. \]
\end{lemma}

\begin{theorem}[\cite{M01,M28}] \label{th:gen sym}
Let $q \geq 3$ be an integer.
\begin{description}
\item[(i)] Except for $q \in \{5,6,8 \}$, $\Sym_q$ is generated by an element of order $2$ and an element of order $3$.
\item[(ii)] If $p^\prime > 3$ is a prime number dividing $q!$ and $q \neq 2p^\prime -1$, then $\Sym_q$ is generated by an element of order $2$ and an element of order $p^\prime$.
\end{description}
\end{theorem}

\begin{lemma}\label{le:prime}
Let $p$ be a prime number and $A$ a finite set of size $q \geq 2$. Then:
\begin{description}
\item[(i)] If $q \geq 3$ and $p=2$, then $\Rank(\ICA(\mathbb{Z}_2; A)) = 3$.
\item[(ii)] If $q \geq 2$ and $p \geq 3$, or $q=p=2$, then $\Rank(\ICA(\mathbb{Z}_p; A)) = 2$.
\end{description}
\end{lemma}
\begin{proof}
If $q=p=2$, the result follows by Example \ref{ex:p=q=2}. Assume $(p,q) \neq (2,2)$. By Lemma \ref{le:ICA iso},
\[ \ICA(\mathbb{Z}_p; A) \cong W:= (\mathbb{Z}_{p} \wr \Sym_{\alpha}) \times \Sym_q  , \]
where $\alpha:=\alpha(p,q) \geq 2$ is the Moreau's necklace-counting function. We use the basic fact that $\Rank(G/N) \leq \Rank(G)$, for any group $G$ and any normal subgroup $N$ of $G$. Let $U_2$ be the $\Sym_\alpha$-invariant submodule of $(\mathbb{Z}_p)^\alpha$ defined in Lemma \ref{le:sub}. Then $U_2$ is a normal subgroup of $\mathbb{Z}_{p} \wr \Sym_{\alpha}$ such that $(\mathbb{Z}_{p} \wr \Sym_{\alpha})/U_2 \cong \mathbb{Z}_p \times \Sym_{\alpha}$. Now, $\Alt_{\alpha}$ is a normal subgroup of $\mathbb{Z}_p \times \Sym_{\alpha}$ with quotient $\mathbb{Z}_p \times \mathbb{Z}_2$. This implies that there is a normal subgroup $N$ of $\mathbb{Z}_p \wr \Sym_{\alpha}$ with quotient isomorphic to $\mathbb{Z}_p \times \mathbb{Z}_2$. Therefore, $N \times \Alt_q$ is a normal subgroup of $W$ with quotient group isomorphic to $\mathbb{Z}_{p} \times \mathbb{Z}_2 \times \mathbb{Z}_2$. Hence,
\begin{equation}\label{eq:low}
\Rank(\mathbb{Z}_{p} \times \mathbb{Z}_2 \times \mathbb{Z}_2) \leq \Rank(W). 
\end{equation}

Define $z_\alpha$ and $z_q$ as in (\ref{permz}). We shall prove the two cases \textbf{(i)} and \textbf{(ii)}.
\begin{description}
\item[(i)] Suppose that $q \geq 3$ and $p = 2$, so $3 \leq \Rank(W)$ by (\ref{eq:low}). We shall show that $W = \langle v_1, v_2, v_3 \rangle$ where
\begin{align*}
v_1 &:= ( (e^1 ; z_\alpha), \id ),   \\
v_2 &:= ( (e^1 ; (1,2)), z_q ), \\ 
v_3 &:= ((e^0; \id), (1,2) ). 
\end{align*}
The projections of $v_1$, $v_2$ and $v_3$ to $\Sym_q$ generate $\Sym_q$, so it is enough to prove that $v_1$ and 
 \[ (v_2)^{o(z_q)} = \begin{cases}
 ((e^1; (1,2)) , \id ), & \text{ if } o(z_q) = 1 \mod(4) \\
 ((e^2; (1,2)) , \id), & \text{ if } o(z_q) = 3 \mod(4)
  \end{cases} \] 
generate $\mathbb{Z}_2 \wr \Sym_\alpha$. Let $M := \langle v_1, (v_2)^{o(z_q)} \rangle$. We follow a similar strategy as in the proof of Lemma \ref{le:rk2}. Note that the projections of $v_1$ and $(v_2)^{o(z_q)}$ to $\Sym_\alpha$ generate $\Sym_\alpha$. Now, $(\mathbb{Z}_2)^{\alpha} \cap M$ is an $\Sym_\alpha$-invariant submodule of $(\mathbb{Z}_2)^{\alpha}$. 

If $\alpha$ is even, then
\[ (v_1)^{o(z_\alpha)} = (1,0,\dots,0) = e^1 \in (\mathbb{Z}_2)^{\alpha} \cap M, \] 
and so $(\mathbb{Z}_2)^{\alpha} \cap M= (\mathbb{Z}_2)^{\alpha}$ in this case. 

Suppose that $\alpha$ is odd. Then
\[ (v_1)^{o(z_\alpha)} = (1,1, \dots 1 )\in (\mathbb{Z}_2)^{\alpha}  \cap M. \]
Observe that  
\[ (v_2)^{2o(z_q)} = (1,1, 0, \dots, 0) \in (\mathbb{Z}_2)^{\alpha} \cap M .\]
By the $2$-transitivity of $\Sym_\alpha$ we obtain that $(0,1,\dots,1) \in  (\mathbb{Z}_2)^{\alpha} \cap M$. Therefore,
\[ e^1 = (1,1,\dots,1) + (0,1,\dots,1) \in (\mathbb{Z}_2)^{\alpha} \cap M, \] 
and we conclude that $(\mathbb{Z}_2)^{\alpha} \cap M= (\mathbb{Z}_2)^{\alpha}$ in this case as well.

\item[(ii)] Suppose that $q \geq 2$ and $p \geq 3$. Then $2 \leq \Rank(W)$ by (\ref{eq:low}). Observe that (\ref{eq:Moreau}) implies that $\alpha = \frac{q^p-q}{p}$ is always an even number. We shall show that $W = \langle u_1, u_2 \rangle$, where
\begin{equation} \label{generators}
 u_1:=((e^1; z_\alpha ), (1,2)) \ \text{ and } \ u_2:=((e^1; (1,2)), z_q ). 
 \end{equation}
As the projections of $u_1$ and $u_2$ to $\Sym_q$ generate $\Sym_q$, it is enough to show that $(u_1)^2$ and $(u_2)^{o(z_q)}$ generate $\mathbb{Z}_{p} \wr \Sym_{\alpha}$. Let $M:=\langle (u_1)^2, (u_2)^{o(z_q)} \rangle$. The projections of $(u_1)^2$ and $(u_2)^{o(z_q)}$ to $\Sym_\alpha$ generate $\Sym_\alpha$, so, as in the proof of Lemma \ref{le:rk2}, it is enough to show that $(\mathbb{Z}_p)^{\alpha} \leq M$. Observe that $(\mathbb{Z}_p)^{\alpha} \cap M$ is a $\Sym_\alpha$-invariant subspace of $(\mathbb{Z}_p)^{\alpha}$.

We shall show that $(\mathbb{Z}_p)^{\alpha}\cap M$ is a nonzero $\Sym_\alpha$-invariant subspace of $(\mathbb{Z}_p)^{\alpha}$ different from $U_1$ and $U_2$, as given by Lemma \ref{le:sub}, so $(\mathbb{Z}_p)^{\alpha}\cap M=M$. Since $p \geq 3$, it suffices to show that at least one of the following elements of $(\mathbb{Z}_p)^{\alpha}\cap M$ is nonzero:
\begin{align*}
w_1 &:= (u_1)^{2(\alpha-1)} = (2(\alpha-1),0, \dots,0), \\
w_2 &:= (u_2)^{2o(z_q)} = (o(z_q),o(z_q),0,\dots,0).
\end{align*} 
If $q = 2$, then $w_2 = (1,1,0,\dots,0)$ is nonzero, as required. Henceforth, suppose $q \geq 3$.  

For $p > 3$ and $q \not \in \{ 5,6,8 \}$, we may replace $(1,2)$ and $z_q$ by generators of $\Sym_q$ of orders $2$ and $3$, respectively (see Theorem \ref{th:gen sym} \textbf{(i)}), so $w_2 = (3,3,0,\dots,0)$ is nonzero. 

If $q=5$, then $w_2$ is nonzero, except when $p=5$. In this case, by equation (\ref{eq:Moreau}),
\[ \alpha - 1 = \frac{5^5 - 5}{5} - 1 = 623 \neq 0 \mod (5),  \]
so $w_1$ is nonzero. If $q=6$, then $w_2$ is nonzero, except when $p=5$. In this case,
\[ \alpha - 1 = \frac{6^5 - 6}{5} - 1 = 1553 \neq 0 \mod (5), \]
so $w_1$ is nonzero. If $q=8$, then $w_2$ is nonzero, except when $p=7$. In this case,
\[ \alpha - 1 = \frac{8^7 - 8}{7} - 1 = 299591 \neq 0 \mod (7), \]
so $w_1$ is nonzero.

Assume that $p=3$. If $q \geq 5$, then $5 \mid q!$ and, for $q \neq 2\cdot 5 -1 = 9$, we may replace $(1,2)$ and $z_q$ by generators of $\Sym_q$ of orders $2$ and $5$, respectively (see Theorem \ref{th:gen sym} \textbf{(ii)}), so $w_2$ is nonzero. If $q=3$, $q=4$, or $q=9$, then
\begin{align*}
\alpha - 1 &= \frac{3^3 - 3}{3}-1 = 7 \neq 0 \mod (3), \\
 \alpha - 1 &= \frac{4^3 - 4}{3} - 1 = 19 \neq 0 \mod(3), \text{ or } \\
\alpha - 1 &= \frac{9^3 - 9}{3} -1 = 239 \neq 0 \mod (3),
\end{align*}
respectively. Therefore, $w_1$ is nonzero, which completes the proof. 
\end{description}
\end{proof}

Recall that for any integer $n \geq 2$, we denote by $\di(n)$ the number of divisors of $n$ (including $1$ and $n$ itself) and by $\di_+(n)$ the number of even divisors of $n$ (so $\di_+(n)=0$ if and only if $n$ is odd).

\begin{theorem}\label{th:rk-ica}
Let $n \geq 2$ be an integer and $A$ a finite set of size $q \geq 2$.
\begin{description}
\item[(i)] If $n$ is not a power of $2$, then
\[ \Rank( \ICA(\mathbb{Z}_n; A)  ) = \begin{cases}
\di(n) + \di_+(n)  - 1 + \epsilon(n,2) & \text{if } q=2 \text{ and } n \in 2\mathbb{Z}; \\ 
\di(n) + \di_+(n) + \epsilon(n,q), & \text{otherwise;}
\end{cases} \]
where $0 \leq \epsilon(n,q) \leq \di(n) - \di_+(n) - 2$. 
\item[(ii)] If $n = 2^k$, then
\[ \Rank( \ICA(\mathbb{Z}_{2^k}; A)  )  = \begin{cases}
2 \di(2^k) - 2 = 2k & \text{if } q=2; \\
2 \di(2^k) -1 = 2k + 1 & \text{if } q \geq 3.
\end{cases} \]
\end{description}
\end{theorem}

\begin{proof}
Let $d_1, d_2, \dots, d_\ell$ be the non-trivial divisors of $n$, with $\ell = \di(n) - 1$, and let 
\[  \ICA(\mathbb{Z}_n; A) \cong W := (\mathbb{Z}_{d_1} \wr \Sym_{\alpha(d_1,q)}) \times \dots \times (\mathbb{Z}_{d_\ell} \wr \Sym_{\alpha(d_\ell, q )})  \times \Sym_q . \]
Suppose first that $q \neq 2$ or $n$ is odd. Then $\alpha(d_i,q) \geq 2$ for all $i$. As in the proof of Lemma \ref{le:prime}, there is a normal subgroup $U \trianglelefteq \mathbb{Z}_{d_i} \wr \Sym_{\alpha(d_i,q)}$ with quotient group $\mathbb{Z}_{d_i} \times \Sym_{\alpha(d_i,q)}$, and $\Alt_{\alpha(d_i,q)}$ is a normal subgroup of $\mathbb{Z}_{d_i} \times \Sym_{\alpha(d_i,q)}$ with quotient group $\mathbb{Z}_{d_i} \times \mathbb{Z}_2$. Hence, there is a normal subgroup $N_{d_i}$ of $\mathbb{Z}_{d_i} \wr \Sym_{\alpha(d_i, q)}$ with quotient isomorphic to $\mathbb{Z}_{d_i} \times \mathbb{Z}_2$. Therefore, $N_{d_1} \times \dots \times N_{d_\ell}$ is a normal subgroup of $W$ with quotient isomorphic to 
\[ Q := (\mathbb{Z}_{d_1} \times \mathbb{Z}_2) \times \dots \times (\mathbb{Z}_{d_\ell} \times \mathbb{Z}_2) \times\mathbb{Z}_2. \]
If $n$ is odd, then $\gcd(2,d_i) = 1$ for all $i$, so 
\[ Q \cong  \mathbb{Z}_{2 d_1} \times \dots \times \mathbb{Z}_{2d_\ell} \times \mathbb{Z}_2,\]
and $\Rank(Q) = \ell + 1 = \di(n)$ in this case. If $n$ is even, suppose that $d_1, \dots, d_{e}$, with $e=\di_+(n)$, are all the even divisors of $n$. Hence,
\[ Q \cong  \mathbb{Z}_{d_1} \times \dots \times \mathbb{Z}_{d_{e}} \times \mathbb{Z}_{2 d_{e+1}} \times \dots \times \mathbb{Z}_{2d_{\ell}} \times (\mathbb{Z}_2)^{e + 1},  \]
and $\Rank(Q) = \ell + e + 1 = \di(n) + \di_+(n)$. This gives the lower bound for the rank of $W$. 

For the upper bound, we shall use the basic fact that $\Rank(G_1 \times G_2) \leq \Rank(G_1) + \Rank(G_2)$, for any pair of groups $G_1$ and $G_2$. Assume first that $n$ is not a power of $2$ and let $d_\ell$ be an odd prime. Hence, $\Rank\left( (\mathbb{Z}_{d_\ell} \wr \Sym_{\alpha(d_\ell, q )})  \times \Sym_q \right) = 2$ by Lemma \ref{le:prime} \textbf{(ii)}, and $\Rank(\mathbb{Z}_{d_i} \wr \Sym_{\alpha(d_i, q )}) = 2$ for all $i$ by Lemma \ref{le:rk2}. Thus, $\Rank(W) \leq 2\ell = 2\di(n) - 2$. If $n$ is a power of $2$, then $\Rank\left( (\mathbb{Z}_{2} \wr \Sym_{\alpha(2, q )})  \times \Sym_q \right) = 3$ by Lemma \ref{le:prime} \textbf{(i)}, so $\Rank(W) \leq 2\ell + 1 = 2\di(n) - 1$.  

When $q=2$ and $n$ is even, we may assume that $d_\ell = 2$, so $\ICA(\mathbb{Z}_n; A) \cong (\mathbb{Z}_{d_1} \wr \Sym_{\alpha(d_1,2)}) \times \dots \times (\mathbb{Z}_{d_{\ell-1}} \wr \Sym_{\alpha(d_{\ell -1} , 2 )})  \times (\mathbb{Z}_2)^2$. The rest of the proof is similar to the previous paragraphs.
\end{proof}

\begin{corollary}
Let $p$ be an odd prime and $k \geq 1$ an integer. Let $A$ be a finite set of size $q \geq 2$. Then:
\[ \Rank( \ICA(\mathbb{Z}_{2^kp}; A)  ) = \begin{cases}
4k + 1 & \text{if } q=2, \\
4k + 2 & \text{if } q \geq 3.
\end{cases}  \]
\end{corollary}
\begin{proof}
This follows by Theorem \ref{th:rk-ica} \textbf{(i)} because $\di(2^k p) - \di_+(2^k p) - 2 = 0$, so $\epsilon(2^k p,q) = 0$.
\end{proof}


\section{The relative rank of $\ICA(\mathbb{Z}_n ; A)$ in $\CA(\mathbb{Z}_n ; A)$ } \label{relative rank}

For any integer $n \geq 2$, define the \emph{divisibility digraph of $n$} as the digraph with vertices $\mathcal{V} := \{ s \in [n] : s \mid n \}$ and edges $\mathcal{E} := \left\{ (s,t) \in \mathcal{V}^2 : t \mid s \right\}$. Denote $E(n) := \left\vert \mathcal{E} \right\vert$.  

\begin{lemma}\label{le:edges}
Let $n \geq 2$. If $n =p_1^{a_1} p_2^{a_2} \dots p_m^{a_m}$, where $p_i$ are distinct primes, then
\[ E(n) = \frac{1}{2^m}\prod_{i=1}^m (a_i + 1)(a_i+2).\]
\end{lemma}
\begin{proof}
Note that the outdegree of any $s = p_1^{b_1} p_2^{b_2} \dots p_m^{b_m} \mid n$ is
\[ \text{outdeg}(s) = (b_1 +1)(b_2 + 1) \dots (b_m +1). \]
Therefore,
\[  E(n)  = \sum_{s \mid n } \text{outdeg}(s) = \sum_{b_1 = 0}^{a_1} \dots \sum_{b_m = 0}^{a_m}  (b_1 +1)(b_2 + 1) \dots (b_m +1)  = \frac{1}{2^m}\prod_{i=1}^m (a_i + 1)(a_i+2). \]
\end{proof}

In the proof of the following result we shall use the notion of \emph{kernel} of a transformation $\tau : A^n \to A^n$ as the partition of $A^n$ induced by the equivalence relation $\{ (x,y ) \in A^n \times A^n : (x)\tau = (y) \tau \}$.	 

\begin{lemma} \label{le:relative}
Let $n \geq 2$ be an integer and $A$ a finite set of size $q \geq 2$. Then:
\[\Rank( \CA(\mathbb{Z}_n; A) : \ICA(\mathbb{Z}_n; A) ) = \begin{cases}
E(n) - 1 & \text{if } q=2 \text{ and } n \in 2 \mathbb{Z}; \\
 E(n) & \text{otherwise.}
\end{cases}  \]
\end{lemma}
\begin{proof}
Let $\mathcal{O}$ be the set of orbits of $\sigma : A^n \to A^n$, as defined in Lemma \ref{le:basic}. Let $d_1, \dots, d_\ell$ be all the divisors of $n$ ordered as follows
\[ 1 = d_1 < d_2 < \dots < d_{\ell -1} < d_\ell= n. \] 
For $1 \leq i \leq \ell$, let $\alpha_i := \alpha(d_i,q)$ and denote by $\mathcal{O}_i$ the subset of $\mathcal{O}$ of orbits of size $d_i$. Let 
\[B_i := \bigcup_{P \in \mathcal{O}_i} P. \]

Suppose that $q \neq 2$ or $n$ is odd, so $\alpha_i \geq 2$ for all $i$. For any pair of divisors $d_j$ and $d_i$ such that $d_j \mid d_i$, fix $\omega_j \in B_j$ and $\omega_i \in B_i$ in distinct orbits. Denote the orbits that contains $\omega_i$ by $[\omega_i]$. Define idempotents $\tau_{i,j} \in \CA(\mathbb{Z}_n; A )$ in the following way:
\[ (x)\tau_{i,j} := \begin{cases}
 (\omega_{j}){\sigma^k} & \text{ if } x =  (\omega_{i}){\sigma^k}  \\
 x & \text{ if } x \in A^n \setminus [\omega_i].
 \end{cases} \]
Note that $\tau_{i,j}$ collapses $[\omega_i]$ to $[\omega_j]$ and fixes everything else.

We claim that
\[ H:=\left\langle \ICA(\mathbb{Z}_n ; A), \tau_{i,j} : d_j \mid d_i \right\rangle = \CA(\mathbb{Z}_n ; A). \]
Let $\xi \in \CA(\mathbb{Z}_n ; A)$.
For $1 \leq i \leq \ell$, and define
\[ (x)\xi_i := \begin{cases}
(x)\xi & \text{ if } x \in B_i \\
x & \text{otherwise}. 
 \end{cases} \] 
Clearly $\xi_i \in \CA(\mathbb{Z}_n; A)$. By Lemma \ref{le:basic}, we have $(B_i)\xi \subseteq \bigcup_{j \leq i }B_i$, so
\[ \xi = \xi_1 \xi_2 \dots \xi_\ell. \]
We shall prove that $\xi_i \in H$ for all $1 \leq i \leq \ell$. Decompose $\xi_i$ as $\xi_i = \xi_i^{\prime} \xi_{i}^{\prime \prime}$, where $(B_i)\xi_i^{\prime} \subseteq \bigcup_{j < i} B_j$ and $(B_i)\xi_{i}^{\prime \prime} \subseteq B_i$.
\begin{enumerate}
\item We show that $\xi_i^{\prime} \in H$. If $B_i = \cup_{s=1}^{\alpha_i} P_s$ is the decomposition of $B_i$ into orbits, we may write $\xi_i^\prime = \xi_i^\prime \vert_{P_1} \dots \xi_i^\prime \vert_{P_{\alpha_i}}$, where $\xi_i^\prime \vert_{P_s}$ acts as $\xi_i^\prime$ on $P_s$ and fixes everything else. In this case, $Q_s := (P_s)\xi_i^\prime \vert_{P_s}$ is an orbit contained in $B_j$ for some $j<i$. By Lemma \ref{le:ICA iso}, there is $\phi_s \in \Sym_{\alpha_i} \times \Sym_{\alpha_j} \leq \ICA(\mathbb{Z}_n;A)$ such that $\phi_s$ acts as the double transposition $( [\omega_i] , P_s) ([\omega_j] ,  Q_s)$, and 
\[  \xi_i^\prime \vert_{P_s} = \phi_s^{-1} \tau_{i, j} \phi_s \in H. \]

\item We show that $\xi_{i}^{\prime \prime} \in H$. In this case, $\xi_i^{\prime \prime} \in \Tran(B_i)$. In fact, as $\xi_i^{\prime \prime}$ preserves the partition of $B_i$ into orbits, $\xi_i^{\prime \prime} \in \langle \sigma \vert_{B_i} \rangle  \wr \Tran_{\alpha_i}$. As $\alpha_i \geq 2$, the semigroup $\Tran_{\alpha_i}$ is generated by $\Sym_{\alpha_i} \leq \ICA(\mathbb{Z}_n;A)$ together with the idempotent $\tau_{i, i}$. Hence, $\xi_i^{\prime \prime} \in H$.    
\end{enumerate}
This establishes that the relative rank of $\ICA(\mathbb{Z}_n ; A)$ in $\CA(\mathbb{Z}_n ; A)$ is at most $E(n)$.
 
For the converse, suppose that
\[ \left\langle \ICA(\mathbb{Z}_n ; A), U  \right\rangle = \CA(\mathbb{Z}_n ; A), \]
where $\vert U \vert < E(n)$. Hence, we may assume that, for some $d_j \mid d_i$, 
\begin{equation}\label{cond}
 U \cap \langle \ICA(\mathbb{Z}_n; A) , \tau_{i, j} \rangle = \emptyset. 
\end{equation}
By Lemma \ref{le:basic}, there is no $\tau \in \CA(\mathbb{Z}_n ; A)$ such that $(X) \tau \subseteq Y$ for $X \in \mathcal{O}_a$, $Y \in \mathcal{O}_b$ with $d_b \nmid d_a$. This, together with (\ref{cond}), implies that $U$ has no element with kernel of the form 
\[ \left\{ \{ x,y \}, \{ z \} : x \in P, y \in Q, z \in A^n \setminus (P \cup Q) \right\} \]
for any $P \in \mathcal{O}_i$, $Q \in \mathcal{O}_j$. Thus, there is no element in $\left\langle \ICA(\mathbb{Z}_n ; A), U  \right\rangle$ with kernel of such form, which is a contradiction (because $\tau_{i,j}\in \CA(\mathbb{Z}_n ; A)$ has indeed this kernel).

The case when $q=2$ and $n$ is even follows similarly, except that now, as there is a unique orbit of size $2$ in $\mathcal{O}$, there is no idempotent $\tau_{2,2}$.
\end{proof}

Finally, Theorems \ref{main1} and \ref{main2} follow by Theorem \ref{th:rk-ica} and Lemmas \ref{le:preliminar}, \ref{le:prime}, \ref{le:edges} and \ref{le:relative}.

\section{Acknowledgment}

This work was supported by the EPSRC grant EP/K033956/1.

\bibliographystyle{amsplain}

\bibliography{MemorylessBibVer2}
	
\end{document}